\documentclass[12pt,a4paper,reqno]{amsart}

\usepackage{graphicx,psfrag,amssymb,amscd}
\usepackage[vcentering]{geometry}
\newcommand{\field}[1]{\ensuremath{\mathbb{#1}}}
\newcommand{\integer}{\field{Z}}
\newcommand{\naturals}{\field{N}}

\newcommand{\real}{\field{R}}

\newcommand{\complete}[2][n]{\ensuremath{K_{#2}^{#1}}}

\newcommand{\sph}{\ensuremath{\mathcal{S}}}
\newcommand{\disc}{\ensuremath{\mathcal{D}}}

\newcommand{\link}[2]{\ensuremath{\ell k(#1,#2)}}
\newcommand{\linktwo}[2]{\ensuremath{\ell k_2(#1,#2)}}
\newcommand{\ambient}[1][n]{\ensuremath{\real^{2#1+1}}}
\newcommand{\rthree}{\ensuremath{\real^3}}
\newcommand{\fancylink}[1]{{\ensuremath{\mathcal{#1}}}}
\newcommand{\JJ}{\fancylink{J}}
\newcommand{\LL}{\fancylink{L}}
\newcommand{\XX}{\fancylink{X}}
\newcommand{\YY}{\fancylink{Y}}
\newcommand{\ZZ}{\fancylink{Z}}

\DeclareMathOperator{\sign}{sign}
\newcommand{\vsphere}[3][n]{\ensuremath{\sigma_{#1}(#2,#3)}}
\newcommand{\key}[2][n]{\ensuremath{\kappa_{#1}(#2)}}
\newcommand{\keydisc}[3][n]{\ensuremath{\kappa_{#1}(#2,#3)}}
\def\co{\colon\thinspace}

\newcommand{\vect}[1]{\ensuremath{\mathbf{#1}}}
\newcommand{\vectg}[1]{\ensuremath{\boldsymbol{#1}}} % for greek vectors

\newcommand{\flapanetal}{Flapan~\etal~\cite{flapan-mellor-naimi2008}}
\newcommand{\etal}{~et~al.}
\newcommand{\path}{path}

\newtheorem{theorem}{Theorem}[section]
\newtheorem{lemma}[theorem]{Lemma}
\newtheorem{corollary}[theorem]{Corollary}
\newtheorem{proposition}[theorem]{Proposition}
\newtheorem*{lemma*}{Lemma}
\newtheorem*{question}{Question}
\numberwithin{equation}{section}

\theoremstyle{definition}

\newtheorem{definition}[theorem]{Definition}

\theoremstyle{remark}

\begin{document}

\title{Intrinsic linking with linking numbers of specified divisibility}
\author{Christopher Tuffley}
\date{\today}
\address{Institute of Fundamental Sciences, Massey University,
         Private Bag 11 222, Palmerston North 4442, New Zealand}
\email{c.tuffley@massey.ac.nz}


\begin{abstract}
Let $n$, $q$ and $r$ be positive integers, and let \complete{N}\ be the $n$--skeleton of an $(N-1)$--simplex. We show that for $N$ sufficiently large every embedding of \complete{N}\ in \ambient\ contains a link $L_1\cup\cdots\cup L_r$ consisting of $r$ disjoint $n$--spheres, such that the linking number $\link{L_i}{L_j}$ is a nonzero multiple of $q$ for all $i\neq j$.
This result is new in the classical case $n=1$ (graphs embedded in $\real^3$) as well as the higher dimensional cases $n\geq 2$; and since it 
implies the existence of a link  $L_1\cup\cdots\cup L_r$ such that $|\link{L_i}{L_j}|\geq q$ for all $i\neq j$, it also extends a result of Flapan\etal\ from $n=1$ to higher dimensions. Additionally, for $r=2$ we obtain an improved upper bound on the number of vertices required to force a two-component link $L_1\cup L_2$ such that \link{L_1}{L_2}\ is a nonzero multiple of $q$. Our new bound has growth $O(nq^2)$, in contrast to the previous bound of growth $O(\sqrt{n}4^nq^{n+2})$.
\end{abstract}

\maketitle

\section{Introduction}

In the early 1980s Sachs~\cite{sachs1983} and Conway and Gordon~\cite{conway-gordon1983} proved that every embedding of the complete graph $K_6$ in $\real^3$ contains a pair of disjoint cycles that form a non-trivial link --- a fact that is expressed by saying that $K_6$ is \emph{intrinsically linked}. Since then, a number of authors have shown that embeddings of larger complete graphs necessarily exhibit more complex linking behaviour, such as non-split many-component links~\cite{flapan-pommersheim-foisy-naimi2001,fleming-diesl2005}; two component links with linking number large in absolute value~\cite{flapan2002}; and two component links with linking number a non-zero multiple of a given integer~\cite{fleming2007,fleming-diesl2005}. 

Such \emph{Ramsey-type results} for intrinsic linking can also be shown to hold in higher dimensions. Let \complete{N}\ be the $n$--skeleton of an $(N-1)$-simplex, which we call the \emph{complete $n$--complex on $N$ vertices}. Then \complete{2n+4}\ is intrinsically linked, in the sense that every
embedding in \ambient\ contains a pair of disjoint $n$-spheres that
have nonzero linking number~\cite{lovasz-schrijver1998,taniyama2000}; and moreover, the results described above can all be extended to embeddings of sufficiently large complete $n$--complexes in \ambient~\cite{ramsey2013}. 

Flapan, Mellor and Naimi~\cite[Thm~1]{flapan-mellor-naimi2008} have shown
that intrinsic linking of graphs is arbitrarily complex, in the
following sense: Given positive integers $r$ and $\alpha$, every embedding
of a sufficiently large complete graph in $\real^3$ contains an
$r$--component link in which the linking number of each pair of
components is at least $\alpha$ in absolute value. The main goal of this
paper is to prove an analogue of this result in all dimensions, with the condition on the magnitude of the linking numbers replaced by a divisibility condition instead. Namely, we show that, given positive integers $r$ and $q$, every embedding of a sufficiently large complete $n$-complex in \ambient\ contains a link consisting of $r$ disjoint $n$--spheres, in which all pairwise linking numbers are nonzero multiples of $q$. 

This result is new in the classical case $n=1$ as well as the higher dimensional cases $n\geq 2$. Since a nonzero multiple of $q$ has magnitude at least $q$, it also extends the Flapan-Mellor-Naimi result to $n\geq 2$. The techniques used to prove it draw heavily on those of Flapan, Mellor and Naimi (for the construction of many-component links with all pairwise linking numbers nonzero), as well as those of our previous paper~\cite{ramsey2013} (for intrinsic linking with $n\geq 2$, and constructing links with linking numbers divisible by $q$). We also obtain a vastly improved upper bound on the number of vertices required in the case $r=2$. Our new bound has growth $O(nq^2)$, in contrast to the previous best bound~\cite[Thm~1.4]{ramsey2013} of growth $O(\sqrt{n}4^nq^{n+2})$.

We note that Flapan, Mellor and Naimi~\cite[Thm~2]{flapan-mellor-naimi2008} show further that intrinsic linking of complete graphs is arbitrarily complex in an even stronger sense: one can additionally require that the second co-efficient of the Conway polynomial of each
component has absolute value at least $\alpha$ as well. We do not consider knotting of the components in this paper. This is in part for reasons of dimension: knotting of $n$--spheres is a codimension two phenomenon, whereas linking of $n$--spheres occurs in codimension $n+1$. Thus, the only dimension in which we can consider intrinsic knotting and linking simultaneously is the classical case $n=1$, and we have not given this case separate consideration. To our knowledge there are at present no known divisibility results for intrinsic knotting, and we pose the following question:

\begin{question}
Let $q\geq 2$ be a positive integer.
Does there exist $n$ such that every embedding of $K_n$ in $\real^3$ contains a knot with second Conway co-efficient a nonzero multiple of $q$? 
\end{question}

\subsection{Statement of results}

Throughout this paper, an \emph{$r$--component link} means $r$
disjoint oriented $n$-spheres embedded in \ambient. Given a 2-component
link $L_1\cup L_2$ we will write \link{L_1}{L_2}\ for their linking number,
and $\linktwo{L_1}{L_2}$ for their linking number mod two.  For
$\{i,j\}=\{1,2\}$ the integral linking number is given by the homology
class $[L_i]$ in $H_n(\ambient-L_j;\integer)\cong\integer$.

Our main result is as follows:
\begin{theorem}
\label{modq.th}
Let $n$, $q$ and $r$ be positive integers, with $r\geq 2$. For $N$ sufficiently large 
every embedding of \complete{N}\ in $\real^{2n+1}$ contains an 
$r$-component link $L_1\cup\cdots\cup L_r$ such that, for every
$i\neq j$, $\link{L_i}{L_j}$ is a nonzero multiple of $q$. 
\end{theorem}

Since every nonzero multiple of $q$ has absolute value at least $q$, Theorem~\ref{modq.th} immediately gives us the following extension of Theorem~1 of \flapanetal\ to higher dimensions:

\begin{corollary}
Let $n$, $\lambda$ and $r$ be positive integers, with $r\geq 2$. For $N$ sufficiently large 
every embedding of \complete{N}\ in $\real^{2n+1}$ contains an 
$r$-component link $L_1\cup\cdots\cup L_r$ such that, for every
$i\neq j$, $|\link{L_i}{L_j}|\geq\lambda$.
\end{corollary}

The $r=2$ case of Theorem~\ref{modq.th} is proved as Theorem~1.4 of~\cite{ramsey2013}, with an upper bound of growth $O(\sqrt{n}4^nq^{n+2})$ on the number of vertices required. We re-prove this result with a greatly improved bound with growth $O(nq^2)$:
\begin{theorem}
\label{2component.th}
For $r=2$, the conclusion of Theorem~\ref{modq.th} holds for 
\[
N\geq \key{q} = \begin{cases}
          24q^2, & n=1,   \\
          4q^2(2n+4)+n+\left\lceil\frac{4q^2-2}{n}\right\rceil+1, & n\geq 2.
          \end{cases}
\]
In other words, every embedding of \complete{\key{q}}\ in \ambient\ contains a two component link $L_1\cup L_2$ such that the linking number \link{L_1}{L_2}\ is a nonzero multiple of $q$. 
\end{theorem}
We note that the bound of Theorem~\ref{2component.th} is equal to the best known upper bound on the number of vertices required to force the existence of a generalised key ring with $q$ keys (see Flapan~\etal~\cite[Lem.~1]{flapan-mellor-naimi2008} for the case $n=1$ (although they don't state this value explicitly), and Tuffley~\cite[Thm.~1.2]{ramsey2013} for $n\geq 2$). 

\subsection{Overview}

As is the case with most Ramsey-type results on intrinsic linking, 
Theorems~\ref{modq.th} and~\ref{2component.th} are proved by using the connect sum operation to combine simpler links into more complicated ones. To achieve the divisibility condition we will require the building block components to be ``large'', in the sense that they all contain two copies of a fixed suitably triangulated disc. The triangulation will not only need to have many $n$--simplices, but must also have a combinatorial structure analogous to a path in a graph. Accordingly, we call such a triangulated disc an \emph{$n$--path}. We give a precise definition of a path in Section~\ref{large.sec}, and then re-establish a number of known results on intrinsic linking to show that we can require the necessary components to be large in this sense. 

The bulk of the work required to prove Theorem~\ref{modq.th} is done in Proposition~\ref{stitchinglinks.prop}, which forms the main technical lemma of the paper. Section~\ref{technical.sec} is devoted to the proof of this. The proposition plays the role of Flapan, Mellor and Naimi's Lemma 2, and the statement and proof are heavily modelled on theirs, making modifications as needed for it to work in all dimensions and achieve the divisibility condition. From an arithmetic standpoint, realising the divisibility condition largely boils down to repeatedly applying the following simple number-theoretic observation, used by both Fleming~\cite{fleming2007} and Tuffley~\cite{ramsey2013}:
\begin{quotation}
\emph{Let $\ell_1,\ell_2,\ldots,\ell_q$ be integers. Then there exist $0\leq a<b\leq q$ such that 
\[
\sum_{i=a+1}^b\ell_i\equiv 0\bmod q.
\]}
\end{quotation}
The work then is in achieving this sum topologically, with the integers involved linking numbers with respect to some fixed sphere $S$. Paths and generalised key rings (links in which one component has nonzero linking number with all the others) play crucial roles in this.

With Proposition~\ref{stitchinglinks.prop} established it is a relatively simple matter to prove Theorem~\ref{modq.th}, and we do this in Section~\ref{main.sec}. The underlying argument is essentially that of Flapan, Mellor and Naimi's proof of their Theorem~1, using our Proposition~\ref{stitchinglinks.prop} in place of their Lemma~2, and with some additional considerations to ensure that the building block components are sufficiently large, in the sense described above.

Finally, we turn our attention to the two component case in Section~\ref{2component.sec}, and establish the improved bound of Theorem~\ref{2component.th}. This is done by simply improving the construction of the building block link used in our original proof~\cite[Thm~1.4]{ramsey2013} of this result. This building block is a generalised key ring with $q$ keys that are all sufficiently large, and our original approach was to obtain this by working with a subdivision of $\complete{N}$. By taking the subdivision fine enough, we could ensure that each key contained the required pair of paths. However, Lemma~\ref{adddisc.lem} gives us a simple way to enlarge the keys of an existing key ring, thereby eliminating the need to subdivide. This by itself dramatically reduces the number of vertices required. By additionally ``recycling'' vertices left over from earlier stages of the construction, we show that we can in fact do this using no more vertices than were needed to construct the initial key ring with $q$ keys, reducing the number of vertices still further.

\subsection{Some notation and terminology}

The combinatorial structure of a link with many components is usefully described by its \emph{linking pattern:}

\begin{definition}[Flapan~\etal~{\cite[Def'ns 1 and 2]{flapan-mellor-naimi2008}}]
Given a link \LL, the \emph{linking pattern} of \LL\ is the graph with vertices the components of \LL, and an edge between two components $K$ and $L$ if and only if $\link{K}{L}\neq0$. 
The \emph{mod 2 linking pattern} of \LL\ is the graph with vertices the components of \LL, and an edge between two components $K$ and $L$ if and only if $\linktwo{K}{L}\neq0$. 
\end{definition}

An $(r+1)$--component link $R\cup L_1\cup\cdots\cup L_r$ is a \emph{generalised key ring} with ring $R$ and keys $L_1,\ldots,L_r$ if its linking pattern contains the star on $r+1$ vertices as a subgraph, with $R$ as the central vertex. Thus, the components $L_i$ all link $R$, just like the keys on a key ring. The link is referred to as a ``generalised'' key ring to reflect the fact that the keys may link each other, which is not typically the case with the kinds of key rings we carry on our persons.

The linking numbers between components of two disjoint many-component links are conveniently collected into a \emph{linking matrix} as follows:
\begin{definition}
Given disjoint ordered oriented links $\JJ=J_1\cup\cdots\cup J_s$, $\LL=L_1\cup\cdots\cup L_t$, we define their \emph{linking matrix} $\link{\JJ}{\LL}$ to be the $s\times t$ matrix with $(i,j)$--entry $\link{J_i}{L_j}$.
\end{definition}
We will say that a matrix $A$ is \emph{positive} if all entries of $A$ are positive, and \emph{nonvanishing} if every entry of $A$ is nonzero.

\section{Constructing links with large components}
\label{large.sec}

A common strategy in proving Ramsey-type results for intrinsic linking is to start with a link with many components and relatively simple linking behaviour, and combine some of the components to form a link with fewer components but more complicated linking behaviour. Our arguments to prove Theorem~\ref{modq.th} will require that the building block linking components are ``large'' in a suitable sense. Thus, in this section we re-establish a number of known results on intrinsic linking to prove the existence of links with large components.

In the classical one-dimensional case (graphs embedded in \rthree) we will simply require our components to have sufficiently many vertices (equivalently, sufficiently many edges). In principle, no additional work is required in this case, because we could simply take a sufficiently large complete graph and subdivide each edge into a suitably long path, as is done in Flapan~\cite{flapan2002}. The combinatorics of triangulated $n$--spheres are more complicated for $n\geq 2$, however, and it will not be sufficient to simply work with spheres with many vertices or $n$--simplices. Instead, we will additionally require our components to be large in the following sense, where $D$ is chosen in advance:
\begin{definition}
\label{large.defn}
Let $D$ be an $n$--dimensional triangulated disc. A triangulated $n$--sphere is \emph{large with respect to $D$} or $D$--\emph{large} if it contains two disjoint oppositely oriented copies of $D$.
\end{definition}
When it comes time to prove Theorem~\ref{modq.th} we will choose $D$ so that it has a triangulation of the following form:
\begin{definition}
\label{path.defn}
Let $D$ be an $n$--dimensional triangulated disc with $\ell$ $n$--simplices. Then $D$ is a \emph{\path\ of length $\ell$} if its $n$--simplices may be labelled $\Delta_1,\ldots,\Delta_\ell$ such that
\[
D_{ab} = \bigcup_{i=a}^b \Delta_i
\]
is a disc for any $1\leq a\leq b\leq \ell$.  
\end{definition}
For $n=1$ this definition co-incides with the usual meaning of a path in a graph. To construct a path for $n\geq2$  we may start with $\ell$ $n$--simplices $\Delta_1,\ldots,\Delta_\ell$, and choose distinct $(n-1)$--simplices $\gamma_i,\delta_i$ belonging to $\Delta_i$. Choose simplicial isomorphisms $\phi_i:\delta_i\to\gamma_{i+1}$ for $1\leq i\leq\ell-1$, and glue the $\Delta_i$ according to the $\phi_i$. The result is a disc $D^n$, and the triangulation $D^n=\Delta_1\cup\cdots\cup\Delta_\ell$ satisfies Definition~\ref{path.defn} by construction. In Lemma~2.6 of~\cite{ramsey2013} it is shown that a disc constructed in this way has $\ell+n$ vertices, and the number of $(n-1)$--simplices in $\partial D^n$ is $\ell(n-1)+2$. We note that for $n\geq2$ a path does not necessarily have this form: for instance, for $n=2$ the triangulation of a regular $n$--gon by radii may be given the structure of a path. 

We begin by establishing the existence of $D$--large $n$--spheres with arbitrarily many additional $n$--simplices. For convenience, we let \vsphere{D}{m}\ be the minimal number of vertices of a triangulated sphere satisfying the conditions of the following lemma.

\begin{lemma}
\label{triangulatedsphere.lem}
Let $D$ be a triangulated disc, and let $m$ be a positive integer. There is a triangulation of $S^n$ that contains two disjoint oppositely oriented copies of $D$, together with at least $m$ additional $n$--simplices.
\end{lemma}

\begin{proof}
Consider $D\times I$. If $V=\{v_0,\ldots,v_N\}$ is the vertex set of $D$, then $D\times I$ has a triangulation with vertex set $V\times\{0,1\}$, and simplices of the form 
\[
\delta_j=[(v_{i_0},0),\ldots,(v_{i_j},0),(v_{i_j},1),\ldots,(v_{i_k},1)]
\] 
for $0\leq j\leq k$ and each $k$-simplex $\delta=[v_{i_0},\ldots,v_{i_k}]$ of $D$ with $i_0<i_1<\cdots<i_k$. 
As a first pass we let $S=\partial(D\times I)$ with the induced triangulation.

The $n$--sphere $S$ contains two disjoint copies of $D$, namely $D\times\{0\}$ and $D\times\{1\}$, and they are oppositely oriented because they are exchanged by reflection in the equator $\partial D\times \{\frac12\}$. Suppose that $\partial D$ contains a total of $t$ simplices of dimension $n-1$. 
Each contributes a total of $n$ simplices of dimension $n$ to $\partial D\times I$, so $S$ has a total of $nt$ additional $n$--simplices. If $nt\geq m$ does not hold then let $S'$ be a triangulated $n$--sphere with at least $m-nt+1$ simplices of dimension $n$ (such a triangulated sphere certainly exists, for example by taking the boundary of a sufficiently long $(n+1)$--path, as constructed above). Choose $n$--simplices $\delta$ and $\delta'$ belonging to $\partial D\times I$ and $S'$, respectively, and form the connected sum of $S$ and $S'$ by gluing the discs $S-\delta$ and $S'-\delta'$ along their boundaries. The resulting sphere satisfies the conditions given in the conclusion of the lemma.
\end{proof}

We now use Lemma~\ref{triangulatedsphere.lem} to prove the existence of generalised key rings with large rings. To do this we require the following slight strengthening of Lemma~3.2 of~\cite{ramsey2013}, which is in turn an extension of Lemma~1 of \flapanetal\ to all dimensions. 

\begin{lemma}
\label{makekeyring.lem}
Let $D$ be a triangulated disc. Suppose that \complete{N}\ is embedded in \ambient\ such that it contains
a link 
\[
L\cup J_1\cup\cdots\cup J_{m^2}\cup X_1\cup\cdots\cup X_{m^2},
\]
where $\linktwo{J_i}{X_i}=1$
for all $i$, and $L$ contains two disjoint oppositely oriented copies of $D$ and at least $m^2$ additional $n$-simplices. 
Then there is an $n$-sphere $Z$ in \complete{N}\ 
with all its vertices on $L\cup J_1\cup\cdots\cup J_{m^2}$, and an
index set $I$ with $|I|\geq\frac{m}2$, such that $\linktwo{Z}{X_j}=1$ for all $j\in I$ and $Z$ contains two disjoint oppositely oriented copies of $D$.
\end{lemma}

In Lemma~3.2 of~\cite{ramsey2013} we require only that $L$ has at least $m^2$ $n$--simplices. Thus, the difference between the two results is the stronger condition that $L$ contains the two copies of $D$ and a further $m^2$ $n$--simplices, and the additional conclusion that $Z$ contains two disjoint oppositely oriented copies of $D$. To prove the stronger form it is only necessary to observe that in proving the original result we can ensure that the copies of $D$ in $L$ end up in $Z$.

\begin{proof}
The first step in the proof of~\cite[Lemma 3.2]{ramsey2013} is to construct an $n$-sphere $\sph$ with all its vertices on $L\cup J_1\cup\cdots\cup J_{m^2}$, and meeting each sphere $J_i$ in an $n$--simplex $\delta_i$. This is done by choosing a distinct $n$--simplex $\delta_i'$ belonging to $L$ for each $i=1,\ldots,m^2$, and applying~\cite[Corollary~2.2]{ramsey2013} to obtain a sphere $Q_i\subseteq\complete{N}$ with all its vertices on $\delta_i\cup\delta_i'$, and meeting $J_i$ in $\delta_i$ and $L$ in $\delta_i'$. The sphere \sph\ is then constructed from $L$ and the $Q_i$ by omitting the interiors of the discs $\delta_i'$. Thus, we can ensure that \sph\ contains two disjoint oppositely oriented copies of $D$ by choosing the  $\delta_i'$ from among the $m^2$ additional $n$--simplices of $L$, leaving the copies of $D$ intact.

At the final step in the proof of~\cite[Lemma 3.2]{ramsey2013}, the required sphere $Z$ is constructed from $S$ and a (possibly empty) subset of the $J_i$, by omitting the interiors of the corresponding $n$--simplices $\delta_i$. Therefore, since \sph\ contains the required copies of $D$, we are guaranteed that $Z$ does too. 
\end{proof}

\begin{corollary}
\label{keyring.cor}
Let $D$ be a triangulated disc, and $r$ a positive integer. For $N$ sufficiently large, every embedding of \complete{N}\ in \ambient\ contains an $(r+1)$--component link $R\cup L_1\cup\cdots\cup L_r$ such that $\linktwo{R}{L_i}=1$ for all $i$, and $R$ contains two disjoint oppositely oriented copies of $D$. It suffices to take
\[
N\geq\keydisc{D}{r}=4r^2(2n+4)+\vsphere{D}{4r^2}.
\]
\end{corollary}

\begin{proof}
Given an embedding of \complete{\keydisc{D}{r}}\ in \ambient, choose $4r^2$
disjoint copies of \complete{2n+4}\ contained in the embedding,
together with a copy of
\complete{\vsphere{D}{4r^2}}. By
Taniyama~\cite{taniyama2000} the $i$th copy of \complete{2n+4}\
contains a 2-component link $J_i\cup X_i$ such that $\linktwo{J_i}{X_i}=1$,
and the copy of $\complete{\vsphere{D}{4r^2}}$ contains a triangulated sphere $L$ that contains two disjoint oppositely oriented copies of $D$ and at least $4r^2$ additional $n$--simplices. 
The result now follows by applying
Lemma~\ref{makekeyring.lem} with $m=2r$ to the link
\[
L\cup J_1\cup\cdots\cup J_{4r^2}\cup X_1\cup\cdots\cup X_{4r^2}.
\qedhere
\]
\end{proof}

Finally, we extend Proposition~1 of \flapanetal\ to higher dimensions, with the additional conclusion that all components are large with respect to a chosen triangulated disc $D$. This result serves as the base case for the inductive argument proving Theorem~\ref{modq.th} in Section~\ref{main.sec}.

\begin{proposition}
\label{bipartite.prop}
Let $D$ be a triangulated disc, and let $r$ be a positive integer. For $N$ sufficiently large, every embedding of \complete{N}\ in \ambient\ contains a $2r$--component link
\[
J_1\cup\cdots\cup J_r\cup L_1\cup\cdots\cup L_r,
\]
such that $\linktwo{J_i}{L_j}$ is nonzero for all $i$ and $j$, and each component contains two disjoint oppositely oriented copies of $D$. 
\end{proposition}

The link given by this result has mod two linking pattern containing the complete bipartite graph $K_{r,r}$, because each component $J_i$ has nonzero mod 2 linking number with each component $L_j$. The argument to prove the existence of such a link is exactly that of Flapan\etal's proof of their Proposition~1, and the extension to higher dimensions already follows from our paper~\cite{ramsey2013}: as noted in Section~1.2.2 of~\cite{ramsey2013} their Proposition~1 is a purely combinatorial argument that depends only on their Lemma~1 and the existence of generalised key rings, and these are generalised to higher dimensions in~\cite{ramsey2013}. So the work to be done here is to ensure that each component contains copies of the disc $D$. 

For $n=1$ this already follows from Flapan\etal's Proposition~1, because we may simply subdivide each edge of a sufficiently large complete graph into paths of length $\ell$. A similar approach could be taken in higher dimensions, using the subdivisions of \complete{N}\ constructed in~\cite{ramsey2013}, but this introduces many unnecessary vertices. We give a simpler argument that doesn't make use of subdivision, and requires far fewer vertices.

\begin{proof}
Following \flapanetal\ let $m=\frac{(4r)^{2^r}}{4}$, and let
\[
N = m\keydisc{D}{r}+r\vsphere{D}{m}.
\]
Then $\complete{N}$ contains $m$ copies of $\complete{\keydisc{D}{r}}$ and $r$ copies of $\complete{\vsphere{D}{m}}$, all disjoint from one another. Given an embedding of \complete{N}\ in \ambient, by Corollary~\ref{keyring.cor} the $i$th copy of  $\complete{\keydisc{D}{r}}$ contains a generalised key ring
\[
R_i\cup J_{i1}\cup\cdots\cup J_{ir}
\]
such that the ring $R_i$ is $D$-large; and the $j$th copy of $\complete{\vsphere{D}{m}}$ contains a $D$-large sphere $L_j$ that contains at least $m$ additional $n$-simplices.

Apply Lemma~\ref{makekeyring.lem} to the link
\[
L_1\cup J_{11}\cup\cdots\cup J_{m1}\cup R_1\cup\cdots\cup R_m.
\]
This yields a $D$-large sphere $Z_1$ with all its vertices on $L_1\cup J_{11}\cup\cdots\cup J_{m1}$, and an index set $I_1$ with $|I_1|\geq\frac{\sqrt{m}}{2}=\frac{(4r)^{2^{r-1}}}{4}=m_1$, such that $\linktwo{Z_1}{R_i}=1$ for all $i\in I_1$. Suppose now that for some $1\leq k<r$ we have constructed $D$-large spheres $Z_1,\ldots,Z_k$ and an index set $I_k$ such that
\begin{enumerate}
\item\label{vertices.item}
all vertices of $Z_j$ lie on $L_j\cup J_{1j}\cup\cdots\cup J_{mj}$ for $1\leq j\leq k$;
\item
$|I_k|\geq m_k = \frac{(4r)^{2^{r-k}}}{4}$;
\item\label{linktwo.item}
$\linktwo{Z_j}{R_i}=1$ for all $1\leq j\leq k$ and $i\in I_k$. 
\end{enumerate}
Applying Lemma~\ref{makekeyring.lem} to the link
\[
L_{k+1}\cup\left(\bigcup_{i\in I_k} J_{i(k+1)}\right)\cup\left(\bigcup_{i\in I_k} R_{i}\right)
\]
we obtain a $D$-large sphere $Z_{k+1}$ with all its vertices on 
$L_{k+1}\cup J_{1(k+1)}\cup\cdots\cup J_{m(k+1)}$, and an index set $I_{k+1}\subseteq I_k$ with $|I_{k+1}|\geq\frac{\sqrt{m_k}}{2}= \frac{(4r)^{2^{r-k-1}}}{4}=m_{k+1}$, such that $\linktwo{Z_{k+1}}{R_i}=1$ for all $i\in I_{k+1}$. This gives us $D$-large spheres $Z_1,\ldots,Z_{k+1}$ and an index set $I_{k+1}$ such that conditions~\eqref{vertices.item}--\eqref{linktwo.item} hold with $k$ replaced by $k+1$, so by induction there are $D$-large spheres $Z_1,\ldots,Z_r$ and an index set $I_r$ such that they hold for $k=r$. Since $m_r=\frac{(4r)^{2^{r-r}}}{4}=r$, the first $2r$ components of 
\[
Z_1\cup\cdots\cup Z_r\cup\left(\bigcup_{i\in I_r} R_{i}\right)
\]
are the required link.
\end{proof}

\section{The main technical lemma}
\label{technical.sec}
 
This section is dedicated to proving the following analogue of Lemma~2 of 
\flapanetal, which forms the main technical lemma of this paper:
\begin{proposition}[Main technical lemma]\label{stitchinglinks.prop}
Let $q\in\naturals$. Suppose that \complete{N}\ is
embedded in \ambient\ such that it contains a link with
oriented components $J_1,\ldots,J_A$, $L_1,\ldots,L_B$,
$X_1,\ldots,X_S$ and $Y_1,\ldots,Y_T$ satisfying
\begin{enumerate}\newcounter{tempcounter}
\item
$A\geq  2^{S}q^{S+T}$;
\item
$B\geq  3^S 2^{T}(S+T)q^{S+T}$;
\item
$\link{J_a}{X_s}$ is nonzero for all $a$ and $s$; 
\item
$\link{L_b}{Y_t}$ is nonzero for all $b$ and $t$; and
\item
\label{existenceofdiscs.item}
each component $J_a,L_b$ contains two disjoint oppositely oriented copies of a fixed path \disc\ of length $\lambda\geq(2q)^{S+T}$.
\end{enumerate}
Then \complete{N}\ contains an $n$-sphere $Z$ with all its vertices
on $J_1\cup\cdots\cup J_A\cup L_1\cup\cdots\cup L_B$ such that,
for each $s$ and $t$, $\link{Z}{X_s}$ and
$\link{Z}{Y_t}$ are nonzero multiples of $q$. 
\end{proposition}

We note that the hypotheses of our Proposition~\ref{stitchinglinks.prop} are much stronger than the hypotheses of Flapan~\etal's Lemma~2: we require $A$ and $B$ to be much greater, and we have the additional hypothesis~\eqref{existenceofdiscs.item} that the components $J_a,L_b$ are large with respect to a certain path. This is to be expected, since our conclusion is strictly stronger than theirs: any nonzero multiple of $q$ is necessarily at least $q$ in magnitude.

Before proving Proposition~\ref{stitchinglinks.prop} we first establish the following lemma on sums of vectors in $\real^d$, which we will use in the proof.

\begin{lemma}\label{forbiddenvalues.lem}
Let $\vect{f}\in\real^d$ be a vector with all entries nonzero, and for $i=0,\ldots,N$ let $\vect{v}_i\in\real^d$. If $N\geq 2^d$ then there exist $0\leq j<k\leq N$ such that every entry of $\vect{f}+\vect{v}_k-\vect{v}_j$ is nonzero.
\end{lemma}

\begin{proof}
The proof is by induction on $d$. In the base case $d=1$, suppose that $N\geq 2$. If either $f+v_1-v_0$ or $f+v_2-v_1$ is nonzero then we are done, and otherwise
\[
f+v_2-v_0 = (f+v_2-v_1)+(f+v_1-v_0)-f=-f\neq 0.
\]
Thus the lemma holds in the base case $d=1$.

Suppose now that the lemma holds for some $d\geq1$, and let $\vect{v}_0,\vect{v}_1,\ldots,\vect{v}_N$ be $N+1\geq 2^{d+1}+1$ vectors in $\real^{d+1}$. We claim that there is $N'\geq 2^d$ and $N'+1$ indices $0\leq i_0<i_1<\cdots<i_{N'}\leq N$ such that, for any $0\leq j<k\leq N'$, the $(d+1)$th entry of $\vect{f}+\vect{v}_{i_k}-\vect{v}_{i_j}$ is nonzero. The inductive step will then follow by applying the inductive hypothesis to the first $d$ entries of $\vect{f}$ and 
$\vect{v}_{i_0},\ldots,\vect{v}_{i_{N'}}$. 

Write $x^{(i)}$ for the $i$th entry of $\vect{x}\in\real^m$. To prove the claim we consider the graph with vertex set $\{0,1,\ldots,N\}$, and an edge between $j$ and $k$ if $j<k$ and the difference $v_k^{(d+1)}-v_j^{(d+1)}$ is equal to the forbidden value $-f^{(d+1)}$. Now observe that for any path $(i_0,i_1,\ldots,i_m)$ in this graph we have
\[
v_{i_m}^{(d+1)}-v_{i_0}^{(d+1)}
=\sum_{j=1}^{m}[v_{i_j}^{(d+1)}-v_{i_{j-1}}^{(d+1)}]
= -f^{(d+1)}\sum_{j=1}^{m-1} \sign(i_{j}-i_{j-1}).
\]
In particular, if the path is a cycle then $i_m=i_0$, and it follows that
$f^{d+1}\sum_{j=1}^{m-1} \sign(i_{j+1}-i_j)=0$. Since $f^{(d+1)}$ is nonzero by hypothesis the sum must be zero, and since each term is $\pm1$, for this to occur it must involve an even number of terms. Thus any cycle must be of even length, and it follows that our graph is bipartite. 

Colour the vertices black and white in such a way that there is no edge between vertices of the same colour, and let $0\leq i_0<i_1<\cdots<i_{N'}\leq N$ be the vertices belonging to the larger colour class. Then $N'+1\geq \lceil (N+1)/2\rceil \geq \lceil(2^{d+1}+1)/2\rceil = 2^d+1$, and for any $0\leq j<k\leq N'$ we have $f^{(d+1)}+v_{i_k}^{(d+1)}-v_{i_j}^{(d+1)}\neq0$, as required. Lemma~\ref{forbiddenvalues.lem} now follows by our discussion above.
\end{proof}

\begin{proof}[Proof of Proposition~\ref{stitchinglinks.prop}]
Let
\begin{align*}
\JJ&=J_1\cup\cdots\cup J_A, & \XX&=X_1\cup\cdots\cup X_S, \\
\LL&=L_1\cup\cdots\cup L_B, & \YY&=Y_1\cup\cdots\cup Y_T. 
\end{align*}
Following \flapanetal, we begin by replacing the links
$\mathcal{J}$ and $\mathcal{L}$ with sublinks $\mathcal{J}'$, $\mathcal{L}''$ for which we have some control over the signs of the entries of the linking matrices $\link{\JJ'}{\XX}$, $\link{\LL''}{\YY}$ and $\link{\LL''}{\XX}$. To do this, we first consider the patterns of signs of the entries of the vectors $\link{J_a}{\XX}$. Since these vectors have $S$ entries, and all are nonzero, there are $2^S$ possibilities for the patterns of signs (positive and negative) in each one. It follows that we can choose at least $A/2^S\geq q^{S+T}$ of them that all have the same pattern of signs. Moreover, after reversing the orientation of some components of \XX\ if necessary, we may assume that these signs are all positive. Thus, setting $\JJ'=J_1\cup\cdots\cup J_{q^{S+T}}$, we may assume without loss of generality that the linking matrix \link{\JJ'}{\XX}\ is positive.

Applying the same argument to the vectors \link{L_b}{\YY}, we obtain a sublink $\LL'$ of $\LL$ with at least $3^S(S+T)q^{S+T}$ components such that the linking matrix \link{\LL'}{\YY}\ is positive. We now consider the patterns of signs (positive, negative or zero) of the vectors \link{L_b}{\XX}\ for $L_b$ a component of $\LL'$. There are now $3^S$ possibilities for these patterns, so we may choose at least $(S+T)q^{S+T}$ components that have the same pattern. Setting $\LL''=L_1\cup\cdots\cup L_{(S+T)q^{S+T}}$ we may therefore assume without loss of generality that the linking matrix \link{\LL''}{\YY}\ is positive, and that each column of \link{\LL''}{\XX}\ is either positive, negative, or zero. From now on we restrict our attention to the sublinks $\JJ'$ and $\LL''$ of \JJ\ and \LL. 

Our next goal is to construct a sublink $\ZZ=Z_1\cup\cdots\cup Z_C$ of
$\JJ'\cup\LL''$ such that every entry of
\[
\vect{z}=\sum_{c=1}^C \link{Z_c}{\XX\cup\YY}
\]
is a nonzero multiple of $q$. At the final step we will obtain the required $n$-sphere $Z$ as a connect sum of the components of \ZZ. To this end we begin by considering the sums
\[
\vect{j}_\alpha = \sum_{a=1}^\alpha \link{J_a}{\XX\cup\YY}
\]
modulo $q$ for $1\leq \alpha\leq q^{S+T}$. Each vector $\vect{j}_\alpha$ has $S+T$ entries, so there are $q^{S+T}$ possibilities when considered mod $q$. Since we have $q^{S+T}$ vectors in total, by the pigeonhole principle we can either find one that is zero modulo $q$, or two that are equal modulo $q$. In either case, there are integers $0\leq \alpha_0< \alpha_1\leq q^{S+T}$ such that the vector
\[
\vect{j} = \sum_{a=\alpha_0+1}^{\alpha_1} \link{J_a}{\XX\cup\YY}
\] 
is zero modulo $q$. Moreover, the first $S$ entries of $\vect{j}$ are given by $\sum_{a=\alpha_0+1}^{\alpha_1} \link{J_a}{\XX}$, and are therefore nonzero, because the vector  \link{J_a}{\XX}\ is positive for each $a$. We will use $J_{\alpha_0+1}\cup\cdots\cup J_{\alpha_1}$ as the first $\alpha_1-\alpha_0$ components of \ZZ.

We now consider the sums
\[
\sum_{b=1}^\beta \link{L_b}{\XX\cup\YY}
\]
modulo $q$ for $1\leq\beta\leq (S+T)q^{S+T}$. Since there are again $q^{S+T}$ possibilities mod $q$, and we have $(S+T)q^{S+T}$ sums in total, we can either find $S+T$ of them that are zero mod $q$, or $S+T+1$ of them that are identical mod $q$. In either case, there are integers $0\leq\beta_0<\beta_1<\cdots<\beta_{S+T}\leq(S+T)q^{S+T}$ such that the vectors
\[
\vectg{\ell}_i 
   = \sum_{b=\beta_0+1}^{\beta_i} \link{L_b}{\XX\cup\YY}
\]
are zero modulo $q$. Any additional components of \ZZ\ will be chosen from among $L_{\beta_0+1}\cup L_{\beta_0+2}\cup\cdots\cup L_{\beta_{S+T}}$. 

To choose the remaining components of \ZZ\ we consider the sequence of $S+T+1$ vectors $\vect{j},\vect{j}+\vectg{\ell}_1,\ldots,\vect{j}+\vectg{\ell}_{S+T}$. From above these vectors are all zero when considered modulo $q$, and we claim that it is possible to choose at least one of them that is nonvanishing when considered as an integer vector. To see this, consider first the $(S+t)$--entries for some $1\leq t\leq T$, which are given by
\begin{align*}
j^{(S+t)}&=\sum_{a=\alpha_0+1}^{\alpha_1}\link{J_a}{Y_t}, \\
(j+\ell_i)^{(S+t)}&=\sum_{a=\alpha_0+1}^{\alpha_1}\link{J_a}{Y_t}+\sum_{b=\beta_0+1}^{\beta_i} \link{L_b}{Y_t}.
\end{align*}
Since the linking matrix \link{\LL''}{\YY}\ is positive these form a strictly increasing sequence, and consequently the $(S+t)$--entry vanishes for at most one of our $S+T+1$ vectors.

Next, consider the $s$--entries for some $1\leq s\leq S$, which are given by
\begin{align*}
j^{(s)}&=\sum_{a=\alpha_0+1}^{\alpha_1}\link{J_a}{X_s}, \\
(j+\ell_i)^{(s)}&=\sum_{a=\alpha_0+1}^{\alpha_1}\link{J_a}{X_s}+\sum_{b=\beta_0+1}^{\beta_i} \link{L_b}{X_s}.
\end{align*}
Recall that the first sum is positive, and that each column of the linking matrix \link{\LL''}\XX\ is either positive, negative, or zero. It follows that the above sequence of integers is either constant (in which case it is positive), or it is strictly increasing or strictly decreasing. In any case we again conclude that the $s$--entry vanishes for at most one of our $S+T+1$ vectors. Thus there are at most $S+T$ vectors for which one of the entries vanishes, and so there is at least one for which no entry vanishes, proving the claim. We may then set
\[
\ZZ = Z_1\cup\cdots\cup Z_C 
  = \begin{cases}
    J_{\alpha_0+1}\cup\cdots\cup J_{\alpha_1} & \text{if $\vect{j}$ is 
                       nonvanishing, or} \\
    J_{\alpha_0+1}\cup\cdots\cup J_{\alpha_1}
           \cup L_{\beta_0+1}\cup\cdots\cup L_{\beta_i} & \text{if $\vect{j}+\vectg{\ell}_i$ is nonvanishing}. 
   \end{cases}
\]
With this choice of \ZZ, every entry of
\[
\vect{z}_0=\sum_{c=1}^C \link{Z_c}{\XX\cup\YY}
\]
is a nonzero multiple of $q$, as required. 

Our final task is to obtain the required $n$-sphere as a suitable connect sum of the components of \ZZ.  To do this we will inductively construct oriented spheres $F_1,\ldots,F_{C-1}$ such that, for each $1\leq\gamma\leq C-1$,
\begin{enumerate}
\renewcommand{\labelenumi}{(\alph{enumi})}
\renewcommand{\theenumi}{\alph{enumi}}
\item\label{F-vertices.item}
the vertices of $F_\gamma$ lie on $Z_\gamma\cup Z_{\gamma+1}$ (and so $F_\gamma$ is disjoint from \XX, \YY, and the rest of \ZZ);
\item\label{F-discs.item}
$F_{\gamma-1}\cap Z_\gamma$ and $F_\gamma\cap Z_\gamma$ are disjoint discs, each of which is oppositely oriented by $Z_\gamma$ and $F_{\gamma-1}$ or $F_\gamma$;
\item\label{F-sum.item}
every entry of the vector
\[
\vect{z}_\gamma=\vect{z}_0 + \sum_{i=1}^\gamma \link{F_i}{\XX\cup\YY}
\]
is a nonzero multiple of $q$.
\end{enumerate}
We will then obtain the required sphere $Z$ from the union of \ZZ\ and the $F_c$ by omitting the interiors of the discs $F_c\cap Z_c$ and $F_c\cap Z_{c+1}$. Conditions~\eqref{F-vertices.item} and~\eqref{F-discs.item} imply that $F_c$ and $F_{c'}$ are disjoint for all $c$ and $c'$, and it follows that $Z$ is a connect sum of spheres, and hence itself a sphere. Moreover, as a chain we have $Z=\sum_{c=1}^C Z_c+\sum_{c=1}^{C-1}F_c$, so 
\[
\link{Z}{\XX\cup\YY} = \vect{z}_0 + \sum_{c=1}^{C-1} \link{F_c}{\XX\cup\YY},
\] 
and by condition~\eqref{F-sum.item} every entry of this vector is a nonvanishing multiple of $q$. 

The underlying technique for constructing the spheres $F_c$ comes from the proof of Theorem~1.4 of Tuffley~\cite{ramsey2013}, but additional work is required to ensure that condition~\eqref{F-sum.item} is satisfied. 
By hypothesis~\eqref{existenceofdiscs.item} each sphere $Z_c$ contains two disjoint copies of the \path\ \disc, one of each orientation. 
We begin by labelling these $D_c$ and $D_c'$ in such a way that there is an  orientation reversing simplicial isomorphism $\phi_c\co D_c\to D_{c+1}'$. This may be done inductively: first label the copies of \disc\ contained in $Z_1$ arbitrarily, and then once $D_c$ and $D_c'$ have been chosen, choose $D_{c+1}$ and $D_{c+1}'$ so that $D_{c+1}'$ is oppositely oriented to $D_c$. We will choose the spheres $F_c$ so that the following strengthened form of condition~\eqref{F-vertices.item} holds for $1\leq\gamma\leq C-1$:
\begin{itemize}
\item[(\ref{F-vertices.item}$'$)]
the vertices of $F_\gamma$ lie on $D_\gamma\cup D_{\gamma+1}'$.
\end{itemize}
This condition serves to ensure that $F_{\gamma-1}\cap Z_\gamma$ and $F_\gamma\cap Z_\gamma$ are disjoint, as required by condition~\eqref{F-discs.item}.

Suppose that for some $0\leq c< C-1$ the spheres $F_1,\ldots,F_{c}$ have been constructed so that conditions~(\ref{F-vertices.item}$'$), \eqref{F-discs.item} and~\eqref{F-sum.item} hold for $0\leq\gamma\leq c$. When $c=0$ conditions~(\ref{F-vertices.item}$'$) and~\eqref{F-discs.item} are empty, and condition~\eqref{F-sum.item} is that every entry of $\vect{z}_0$ is a nonzero multiple of $q$, so we may take $c=0$ as our base case. Let $\Delta_1,\ldots,\Delta_\lambda$ be a labelling of the $n$--simplices of the path $D_{c+1}$ as in Definition~\ref{path.defn}, and for $1\leq\ell\leq\lambda$ let $P_\ell$ be the oriented sphere satisfying 
\begin{align*}
P_{\ell} \cap Z_{c+1} &= \Delta_{\ell}, & P_{\ell} \cap Z_{c+2} &= \phi_{c+1}(\Delta_{\ell})
\end{align*}
that results from applying
Corollary~2.2 of Tuffley~\cite{ramsey2013} to the pairs $(Z_{c+1},D_{c+1})$ and
$(Z_{c+2},D_{c+2}')$.
The vertices of these spheres all lie on $D_{c+1}\cup D_{c+2}'$, and for any $1\leq\mu\leq\nu\leq\lambda$, the chain $\sum_{\ell=\mu}^\nu P_\ell$ represents a sphere meeting $D_{c+1}$ in the disc 
$\bigcup_{\ell=\mu}^\nu \Delta_\ell$, and $D_{c+2}'$ in the disc
$\bigcup_{\ell=\mu}^\nu \phi_{c+1}(\Delta_\ell)$.

For $1\leq\ell\leq\lambda$ we consider the sums
\[
\sum_{i=1}^\ell\link{P_i}{\XX\cup\YY}
\]
modulo $q$. As above there are $q^{S+T}$ possibilities for these modulo $q$, and we have $\lambda\geq 2^{S+T}q^{S+T}$ of them, so we can either find $2^{S+T}$ of them that are identically zero mod $q$, or $2^{S+T}+1$ of them that are equal mod $q$. In either case there are integers $0\leq\mu_0<\mu_1<\cdots<\mu_{2^{S+T}}$ such that the vectors
\[
\vect{p}_{j}=\sum_{i=\mu_0+1}^{\mu_j} \link{P_i}{\XX\cup\YY}
\]
are identically zero mod $q$ for $1\leq j\leq2^{S+T}$. 

Set $\vect{p}_0=\vect{0}$, and apply Lemma~\ref{forbiddenvalues.lem} to the vectors $\vect{p}_0,\vect{p}_1,\ldots,\vect{p}_{2^{S+T}}\in\real^{S+T}$ with $\vect{f}=\vect{z}_c$. This yields indices $0\leq j<k\leq 2^{S+T}$ such that no entry of 
\[
\vect{z}_c+\vect{p}_{k}-\vect{p}_j=\vect{z}_c+\sum_{i=\mu_j+1}^{\mu_k} \link{P_i}{\XX\cup\YY}
\]
is zero. Moreover, the vectors $\vect{z}_c$, $\vect{p}_j$ and $\vect{p}_k$ are all identically zero mod $q$, so every entry of 
$\vect{z}_c+\vect{p}_{k}-\vect{p}_j$ is a nonzero multiple of $q$. 

Let $F_{c+1}=\sum_{i=\mu_j+1}^{\mu_k} P_i$.
Then $F_{c+1}$ represents an $n$--sphere with all its vertices on
$Z_{c+1}\cup Z_{c+2}$, and meeting $Z_{c+1}$ and $Z_{c+2}$ in the discs
\begin{align*}
F_{c+1}\cap Z_{c+1} &= \bigcup_{i=\mu_j+1}^{\mu_k} \Delta_i\subseteq D_{c+1}, &
F_{c+1}\cap Z_{c+2} &= \phi_{c+1}\left(\bigcup_{i=\mu_j+1}^{\mu_k} \Delta_i\right)\subseteq D_{c+2}'.
\end{align*}
The construction of Corollary~2.2 of Tuffley~\cite{ramsey2013} ensures that these discs are oppositely oriented by $F_{c+1}$ and $Z_{c+1}\cup Z_{c+2}$, so conditions~(\ref{F-vertices.item}$'$) and~\eqref{F-discs.item} are satisfied; and with this choice of $F_{c+1}$ we have
$\vect{z}_{c+1} = \vect{z}_c+\vect{p}_k-\vect{p}_j$, so condition~\eqref{F-sum.item} is too. This completes the inductive step, and we now obtain the required sphere $Z$ as described above.
\end{proof}

\section{Proof of Theorem~\ref{modq.th}}
\label{main.sec}

We're now in a position to prove our main result, Theorem~\ref{modq.th}. The strategy is that of Flapan\etal's proof of their Theorem~1.

\begin{proof}[Proof of Theorem~\ref{modq.th}]
  Following \flapanetal, for each $u,v\in\naturals$ let $H(u,v)$ denote the complete $(u+2)$--partite graph with parts $P_1$ and $P_2$ containing $v$ vertices each, and parts $Q_1,\ldots,Q_u$ containing a single vertex each. We will prove by induction on $u$ that for every $u\geq 0$ and $v,\ell\geq1$, for $N$ sufficiently large every embedding of \complete{N}\ in \ambient\ contains a link $\LL$ such that
\begin{enumerate}
\renewcommand{\labelenumi}{(L\arabic{enumi})}
\renewcommand{\theenumi}{(L\arabic{enumi})}
\item\label{linkingpattern.item}
the linking pattern of \LL\ contains the graph $H(u,v)$;
\item\label{linkingnumber.item}
the linking number between any two distinct components in $Q_1\cup\cdots\cup Q_u$ is a nonzero multiple of $q$; and
\item\label{path.item}
every component in $P_1\cup P_2$ contains disjoint oppositely oriented copies of a \path\ $D$ of length at least $\ell$.
\end{enumerate}
For simplicity, we will say that a link \LL\ satisfying conditions~\ref{linkingpattern.item}--\ref{path.item} with the given parameter values satisfies property $(u,v,\ell)$. 

The base case $u=0$ follows from Proposition~\ref{bipartite.prop} with $r=v$, by choosing $D$ to be a \path\ of length $\ell$. Suppose then that the claim holds for some $u\geq 0$. Given $v,\ell\geq 0$, let 
\begin{align*}
S     &= v, \\
T     &= u+v, \\
A = B &= 2^{T}3^S(S+T)q^{S+T} 
         \geq 2^Sq^{S+T}, \\
\lambda &= \max\{\ell,(2q)^{S+T}\},
\end{align*}
and let $w=S+A=S+B$. By our inductive hypothesis, for $N$ sufficiently large every embedding of \complete{N}\ in \ambient\ contains a link \LL\ satisfying property $(u,w,\lambda)$. We will show that every such embedding also contains a link $\LL'$ satisfying property $(u+1,v,\ell)$.

Given an embedding of \complete{N}\ in \ambient\ and a link \LL\ contained in it satisfying property $(u,w,\lambda)$, label the components of \LL\ such that
\begin{align*}
P_1 &= \{X_1,\ldots,X_S,L_1,\ldots,L_B\}, \\
P_2 &= \{Y_1,\ldots,Y_S,J_1,\ldots,J_A\},
\end{align*}
and $Q_i=\{Y_{v+i}\}$ for $1\leq i\leq u$. Then all linking numbers $\link{J_a}{X_s}$ and $\link{L_b}{Y_t}$ are nonzero by~\ref{linkingpattern.item}, and every component $J_a$, $L_b$ contains two disjoint copies of a path \disc\ of length at least $\lambda \geq (2q)^{S+T}$, by~\ref{path.item}. So we may apply Proposition~\ref{stitchinglinks.prop} to \LL\ to obtain a sphere $Z$ with all its vertices on $J_1\cup\cdots\cup J_A\cup L_1\cup\cdots\cup L_B$ and linking every component $X_s,Y_t$ with linking number a nonzero multiple of $q$. Let 
\begin{align*}
\LL'&=X_1\cup\cdots\cup X_S\cup Y_1\cup\cdots\cup Y_T\cup Z \\
    &=X_1\cup\cdots\cup X_v\cup Y_1\cup\cdots\cup Y_{u+v}\cup Z,
\end{align*}
and partition the components as $P_1'\cup P_2'\cup Q_1'\cup\cdots\cup Q_{u+1}'$ such that
\begin{align*}
P_1' &= \{X_1,\cdots,X_v\}, \\
P_1' &= \{Y_1,\cdots,Y_v\}, 
\end{align*}
and
\[
Q_i'  = \begin{cases}
        \{Y_{v+i}\} & 1\leq i\leq u, \\
        \{Z\}       & i=u+1.
        \end{cases}
\]
Then with respect to this partition the linking pattern of $\LL'$ contains the graph $H(u+1,v)$; any two components in $Q_1'\cup\cdots\cup Q_{u+1}'$ have linking number a nonzero multiple of $q$; and every component in $P_1\cup P_2$ contains a copy of \disc, which is a path of length at least $\lambda\geq\ell$. So $\LL'$ satisfies property $(u+1,v,\ell)$, completing the inductive step. By~\ref{linkingnumber.item} the result now follows by restricting attention to $Q_1\cup\cdots\cup Q_u$, with $u=r$.
\end{proof}

\section{The two component case}
\label{2component.sec}

We now turn to the two component case, and establish the improved bound of Theorem~\ref{2component.th}. 

From the proof of~\cite[Theorem~1.4]{ramsey2013} it suffices to prove every embedding of \complete{\key{q}}\ contains a generalised key ring with $q$ keys each large with respect to a path $D$ of length $q$. The approach of~\cite{ramsey2013} was to work with a subdivision of \complete{N}, in which each $n$-simplex was subdivided into $q^n$ simplices. This is a fairly extravagant approach, since only $2q$ $n$-simplices from each component are used to form the required paths. 
The reduction in the number of vertices required comes from Lemma~\ref{adddisc.lem}, which gives us a simple and economic way to enlarge the keys of an existing generalised key ring. A further modest saving comes from ``recycling'' some of the vertices leftover from the construction of the initial key ring.

\begin{lemma}
\label{adddisc.lem}
Let \complete{N}\ be embedded in \ambient\ such that it contains a link $X\cup Y$ with $\link{X}{Y}\neq0$. Let $D$ be a triangulated $n$--disc with $d$ vertices, and suppose that $V$ is a set of $2d-(n+1)$ vertices of $\complete{N}$ disjoint from $X\cup Y$. Then \complete{N}\ contains a $D$-large sphere $Z$ with all its vertices on $Y\cup V$ such that $\link{X}{Z}\neq0$.

The result also holds with all linking numbers calculated mod 2.
\end{lemma}

\begin{proof}
Choose an $n$-simplex $\Delta$ belonging to $Y$, and let $S=\partial(D\times I)$ with the triangulation with $2d$ vertices from the proof of
Lemma~\ref{triangulatedsphere.lem}. Then $\Delta\cup V$ contains a total of $(n+1)+(2d-(n+1))=2d$ vertices, so we may embed $S$ in \complete{N}\ such that all vertices of $S$ lie on $\Delta\cup V$ and $\Delta$ is an $n$--simplex of $\partial D\times I$. Orient $S$ such that $\Delta$ receives opposite orientations from $S$ and $Y$, and consider the chains $S$ and $T=S+Y$. Both represent $D$-large $n$--spheres with all their vertices on $Y\cup V$, and the linking numbers $\link{X}{S}$, $\link{X}{T}$ cannot both be zero because in the homology group $H_n(\ambient-X)$ we have
\begin{equation}
\label{T-S.eq}
[T]-[S] = [S+Y]-[S]=[Y]\neq 0.
\end{equation}
We may therefore choose one of $S$ and $T$ to be $Z$ so that $\link{X}Z\neq0$.

If $\linktwo{X}{Y}\neq0$ then equation~\eqref{T-S.eq} holds in $H_n(\ambient-X;\integer/2\integer)$, and we may again choose $Z$ to be one of $S$ and $T$ so that $\linktwo{X}{Z}\neq0$.
\end{proof}

\begin{corollary}
\label{largekeys.cor}
Let $q$ be a positive integer. Then every embedding of \complete{\key{q}}\ in \ambient\ contains a generalised key ring in which each key is large with respect to a path $D$ of length $q$.
\end{corollary}

\begin{proof}
By~\cite[Theorem 1.2]{ramsey2013} every embedding of \complete{\key{q}}\ in \ambient\ contains a generalised key ring \LL\ with $q$ keys. This link is constructed by applying~\cite[Lemma~3.2]{ramsey2013} (the extension of~\cite[Lemma~1]{flapan-mellor-naimi2008} to higher dimensions) to a link
\[
L\cup J_1\cup\cdots\cup J_{4q^2}\cup K_1\cup\cdots\cup K_{4q^2},
\]
in which $\linktwo{J_i}{K_i}$ is nonzero for all $i$, and each component $J_i,K_i$ is the boundary of an $(n+1)$--simplex. This yields an $n$--sphere $R$ with all vertices on $L\cup J_1\cup\cdots\cup J_{4q^2}$ and linking at least $q$ of the $K_i$, which forms the ring of the generalised key ring. 
Let $K_{i_1},\ldots,K_{i_q}$ be the keys. 

Recall that a path $D$ of length $q$ can be constructed using as few as $d=q+n$ vertices. Since only $q$ of the $K_i$ are components of \LL\ this leaves at least $(4q^2-q)(n+2)=q(4q-1)(n+2)$ vertices of \complete{\key{q}}\ that do not belong to \LL. Observe that
\[
(4q-1)(n+2) = (4q-1)n+8q-2\geq 2n+2q = 2d > 2d-(n+1).
\]
The spare vertices are therefore more than enough to apply Lemma~\ref{adddisc.lem} $q$ times to $R$ and each key $K_{i_j}$ in turn, replacing $K_{i_j}$ with a $D$-large sphere $Z_j$ that still links $R$. Then
\[
R\cup Z_1\cup\cdots\cup Z_q
\]
is the desired link.
\end{proof}

For completeness' sake we sketch the steps needed to prove Theorem~\ref{2component.th} from this point. For any missing details see the proof of~\cite[Thm~1.4]{ramsey2013}, or the corresponding step of the proof of Proposition~\ref{stitchinglinks.prop}.

\begin{proof}[Proof of {Theorem~\ref{2component.th}}]
By Corollary~\ref{largekeys.cor}, every embedding of \complete{\key{q}}\ in \ambient\ contains a generalised key ring
$R\cup Z_1\cup\cdots\cup Z_q$ such that each key $Z_i$ is large with respect to a path $D$ of length $q$. Orient the $Z_i$ so that all linking numbers with $R$ are positive. Working in the homology group $H_n(\ambient-R;\integer)$, let $1\leq a\leq b\leq q$ be such that 
\[
\sum_{i=a}^b [Z_i] \equiv 0\bmod q,
\]
and note that this sum is positive. From now on we restrict our attention to the spheres $Z_a,\ldots,Z_b$.

If $a=b$ we are done. Otherwise, we use the fact that each component $Z_i$ is $D$-large to construct oriented spheres $F_a,\ldots,F_{b-1}$ such that, for $a\leq i\leq b-1$,
\begin{enumerate}
\renewcommand{\labelenumi}{(\alph{enumi})}
\renewcommand{\theenumi}{\alph{enumi}}
\item
the vertices of $F_i$ lie on $Z_i\cup Z_{i+1}$ (and so $F_i$ is disjoint from $R$ and the rest of the $Z_j$);
\item
$F_{i-1}\cap Z_i$ and $F_i\cap Z_i$ are disjoint discs, each of which is oppositely oriented by $Z_i$ and $F_{i-1}$ or $F_i$;
\item\label{Fmodq.item}
the linking number $\link{R}{F_i}$ is zero mod $q$. 
\end{enumerate}
The construction of the $F_i$ is identical to that of the corresponding spheres in Proposition~\ref{stitchinglinks.prop}, except that the simpler condition~\eqref{Fmodq.item} means we only require $D$ to have length $q$, and the spheres can all be constructed simultaneously instead of inductively. 
Now if $\link{R}{F_i}$ is nonzero for some $i$ then $R\cup F_i$ is the required link; and otherwise, we let $Z$ be the connect sum of $Z_a,\ldots,Z_b,F_a,\ldots,F_{b-1}$ obtained by omitting the interiors of the discs $F_i\cap Z_i$ and $F_i\cap Z_{i+1}$ for each $i$. Then $Z$ is an $n$--sphere, and in $H_n(\ambient-R)$ we have
\[
[Z] = \sum_{i=a}^b [Z_i] + \sum_{i=a}^{b-1}[F_i] = \sum_{i=a}^b [Z_i],
\]
which is a nonzero multiple of $q$
\end{proof}

\bibliographystyle{plain}
\bibliography{linking}

\begin{thebibliography}{10}

\bibitem{conway-gordon1983}
J.~H. Conway and C.~McA. Gordon.
\newblock Knots and links in spatial graphs.
\newblock {\em J. Graph Theory}, 7(4):445--453, 1983.

\bibitem{flapan2002}
Erica Flapan.
\newblock Intrinsic knotting and linking of complete graphs.
\newblock {\em Algebr. Geom. Topol.}, 2:371--380 (electronic), 2002.
\newblock E-print arXiv:math/0205231v1.

\bibitem{flapan-mellor-naimi2008}
Erica Flapan, Blake Mellor, and Ramin Naimi.
\newblock Intrinsic linking and knotting are arbitrarily complex.
\newblock {\em Fund. Math.}, 201(2):131--148, 2008.
\newblock E-print arXiv:math/0610501v6.

\bibitem{flapan-pommersheim-foisy-naimi2001}
Erica Flapan, James Pommersheim, Joel Foisy, and Ramin Naimi.
\newblock Intrinsically {$n$}-linked graphs.
\newblock {\em J. Knot Theory Ramifications}, 10(8):1143--1154, 2001.

\bibitem{fleming2007}
Thomas Fleming.
\newblock Intrinsically linked graphs with knotted components.
\newblock {\em J. Knot Theory Ramifications}, 21(7):1250065, 10, 2012.
\newblock E-print arXiv:0705.2026.

\bibitem{fleming-diesl2005}
Thomas Fleming and Alexander Diesl.
\newblock Intrinsically linked graphs and even linking number.
\newblock {\em Algebr. Geom. Topol.}, 5:1419--1432 (electronic), 2005.
\newblock E-print arXiv:math/0511133v1.

\bibitem{lovasz-schrijver1998}
L{\'a}szl{\'o} Lov{\'a}sz and Alexander Schrijver.
\newblock A {B}orsuk theorem for antipodal links and a spectral
  characterization of linklessly embeddable graphs.
\newblock {\em Proc. Amer. Math. Soc.}, 126(5):1275--1285, 1998.

\bibitem{sachs1983}
Horst Sachs.
\newblock On a spatial analogue of {K}uratowski's theorem on planar graphs ---
  an open problem.
\newblock In {\em Graph theory (\L ag\'ow, 1981)}, volume 1018 of {\em Lecture
  Notes in Math.}, pages 230--241. Springer, Berlin, 1983.

\bibitem{taniyama2000}
Kouki Taniyama.
\newblock Higher dimensional links in a simplicial complex embedded in a
  sphere.
\newblock {\em Pacific J. Math.}, 194(2):465--467, 2000.

\bibitem{ramsey2013}
Christopher Tuffley.
\newblock Some {R}amsey-type results on intrinsic linking of $n$-complexes.
\newblock {\em Alg. Geom. Topology}, 13(3):1579--1612, 2013.
\newblock E-print arXiv:1112.4558v3.

\end{thebibliography}

\end{document}